\documentclass[10pt]{article}
\usepackage{amsmath, amssymb, amsthm, latexsym, amsfonts, epsfig, color,graphicx}

\begin{document}

\newcommand{\s}{\sigma}
\renewcommand{\k}{\kappa}
\newcommand{\p}{\partial}
\newcommand{\D}{\Delta}
\newcommand{\om}{\omega}
\newcommand{\Om}{\Omega}
\renewcommand{\phi}{\varphi}
\newcommand{\e}{\epsilon}
\renewcommand{\a}{\alpha}
\renewcommand{\b}{\beta}
\newcommand{\N}{{\mathbb N}}
\newcommand{\R}{{\mathbb R}}
   \newcommand{\eps}{\varepsilon}
   \newcommand{\EX}{{\Bbb{E}}}
   \newcommand{\PX}{{\Bbb{P}}}

\newcommand{\cF}{{\cal F}}
\newcommand{\cG}{{\cal G}}
\newcommand{\cD}{{\cal D}}
\newcommand{\cO}{{\cal O}}

\newcommand{\grad}{\nabla}
\newcommand{\n}{\nabla}
\newcommand{\curl}{\nabla \times}
\newcommand{\dive}{\nabla \cdot}

\newcommand{\ddt}{\frac{d}{dt}}
\newcommand{\la}{{\lambda}}

\newcommand{\pp}{\frac{\partial}{\partial t}}
\newcommand{\ppn}{\frac{\partial}{\partial n}}

\newcommand{\F}{{\mathcal{F}}}
\newcommand{\B}{{\mathcal{B}}}
\newcommand{\cN}{{\mathcal{N}}}

\newtheorem{theorem}{Theorem}
\newtheorem{lemma}{Lemma}
\newtheorem{remark}{Remark}
\newtheorem{example}{Example}
\newtheorem{definition}{Definition}

\title{Predictability in Spatially Extended Systems
\\with Model Uncertainty }

\author{Jinqiao Duan \\
Department of Applied Mathematics \\
Illinois Institute of Technology \\
Chicago, IL 60616, USA \\ E-mail:  \emph{duan@iit.edu}   }

  \date{March 26, 2009 (Revised version) }

\maketitle

\begin{abstract}
Macroscopic models for spatially extended systems under random
influences are often described   by stochastic partial
differential equations (SPDEs).
 Some   techniques for understanding solutions of such equations, such as
 estimating correlations, Liapunov exponents and impact of noises, are discussed.
 They are relevant for understanding
  predictability in spatially extended  systems with model uncertainty,  for example, in physics,
  geophysics and biological sciences. The presentation is for a wide audience.

\medskip


\end{abstract}

 \tableofcontents


\section{Motivation}

Scientific and engineering systems are often subject to
uncertainty or random  influence. Randomness can have delicate
impact on the overall evolution of such  systems, for example,
stochastic bifurcation \cite{CarLanRob01}, stochastic resonance
\cite{Imkeller},
 and  noise-induced pattern formation \cite{Gar}.
Taking stochastic effects into account is of central importance
for the  development of mathematical models of complex phenomena
in engineering and science.

Macroscopic models for systems with spatial dependence
(``spatially extended") are often in the form of partial
differential equations (PDEs).  Randomness appears in these models
as stochastic forcing, uncertain parameters, random sources or
inputs, and random boundary conditions (BCs). These models are
usually called stochastic partial differential equations (SPDEs).
Note   that SPDEs may also serve as intermediate ``mesoscopic"
models in some multiscale systems. Although we may think that
SPDEs could be reduced to large systems of stochastic ordinary
differential equations (SODEs) in numerical approaches
\cite{Millet, Allouba}, it   is beneficial to work on SPDEs
directly when dealing with some dynamical issues
\cite{Bong,Cessi,DelSole,DuanGaoSchm,Farrell2,
Griffa,Holloway,Imkeller3,Muller,NorCah81,Samelson,
Saravanan-McWilliams, Sura3}.

There is a growing recognition of a role for the inclusion of
 stochastic terms in the modeling of complex systems.
 For example, there has been increasing interest  in mathematical
 modeling via  SPDEs, for
  the climate system, condensed matter physics,
materials sciences,  mechanical and electrical engineering, and
finance, to name just a few.   The inclusion
 of stochastic  effects has led to interesting new mathematical problems
 at the interface of dynamical systems,  partial
 differential equations, scientific computing, and probability theory.
Problems arising in the context of stochastic dynamical modeling
have
 inspired interesting research  topics about, for example,
 the interaction between noise,
 nonlinearity and multiple scales, and about efficient numerical methods for
 simulating random phenomena.

 There has been some   promising new development in understanding dynamics of
 SPDEs via invariant manifolds
\cite{DuanLuSchm, DuanLuSchm2, MZZ, WangDuan} and stochastic
homogenization \cite{WangDuan1, WangDuan2}. But we will not
discuss these issues in this paper. For general background on
SPDEs, see \cite{DaPrato, Roz, Walsh, Chow, Rockner}.

Although  some   progress has been made in SPDEs in the past
decade, many challenges remain and new problems arise in modeling
basic mechanisms in  complex systems under uncertainty.    These
challenging problems include overall impact of noise, stochastic
bifurcation, ergodic theory, invariant manifolds,   and
predictability of dynamical behavior,   to name just a few.
Solutions for these problems will greatly enhance our ability in
understanding, quantifying, and managing uncertainty and
predictability in engineering and science. Breakthroughs in
solving these challenging problems are expected to emerge.

This article is organized as follows. After reviewing some basic
concepts on probability in Hilbert space in \S
\ref{tools-Hilbert}, we discuss stochastic analysis and  SPDEs in
\S \ref{spde}. Then we derive correlations of some linear SPDEs,
Lyapunov exponents, and the impact of uncertainty in \S
\ref{corr}, \S \ref{liapunov} and \S \ref{impact}, respectively.

\section{Stochastic Tools in Hilbert Space}
\label{tools-Hilbert}

\subsection{Hilbert space}

Recall that the Euclidean space $\R^n$ is equipped  with the usual
metric or distance $d(x,y)=\sqrt{\sum_{j=1}^n (x_j-y_j)^2}$, norm
or length $\|x\| =\sqrt{\sum_{j=1}^n  x_j ^2}$, and the usual
scalar product $ x \cdot y = <x, y>= \sum_{j=1}^n x_j y_j$. The
Borel $\s-$field of  $\R^n$, i.e., $\mathcal{B}(\R^n)$ is
generated by all open balls in $\R^n$.

Hilbert space $H$ is a set with three mathematical operations:
scalar multiplication, addition and scalar product $\langle \cdot,
\cdot \rangle$,   satisfying the \emph{usual} properties as we are
familiar with in elementary mathematics. The scalar product
induces a natural norm $\| u\| = \sqrt{\langle u, u\rangle}$. The
Borel $\s-$field of  $H$, i.e., $\mathcal{B}(H)$ is generated by
all open balls in $H$.

\subsection{Probability in Hilbert space}


Given a probability space $(\Omega, \mathcal{F}, \mathbb{P})$,
with sample space $\Om$, $\s-$field $\mathcal{F}$ and probability
measure $\PX$.  Consider a random variable in Hilbert space $H$
(i.e., taking values in $H$):
$$
X: \Om \to H.
$$
Its mean or mathematical expectation is defined in terms of the
integral with respect to the probability measure $\PX$:
$$
\EX (X) = \int_{\Om} X(\om) d\PX(\om).
$$

Its variance is:
$$
Var (X) =  \EX \langle X-\EX (X),  X-\EX (X)\rangle= \EX \|X-\EX
(X)\|^2 = \EX \|X\|^2 - \|\EX (X)\|^2
$$
Especially, if $\EX (X)=0$, then $ Var(X) =\EX \|X\|^2$.

Covariance operator of $X$ is defined as \index{Covariance
operator in Hilbert space}
\begin{eqnarray}
 Cov(X) = \EX [(X-\EX(X))\otimes(X-\EX(X))],
\end{eqnarray}
where for any $a, b \in H$, we denote $a \otimes b$ the linear
operator in $H$ defined by
\begin{eqnarray}
a \otimes b: \; H \to H, \\
 (a \otimes b) h = a \langle b, h\rangle, \; h \in H.
\end{eqnarray}

Let $X$ and $Y$ be two random variables taking values in Hilbert
space $H$. The correlation operator of $X$ and $Y$ is defined by
\index{Correlation operator in Hilbert space}
\begin{eqnarray}
 Cor(X, Y) = \EX [(X-\EX(X))\otimes(Y-\EX(Y))].
\end{eqnarray}

\begin{remark}
$Cov(X)$ is a symmetric positive and trace-class linear operator
with trace
\begin{eqnarray}
 Tr \; Cov(X) =\EX  \langle X-\EX(X), X-\EX(X)\rangle= \EX \|X-\EX(X)\|^2.
\end{eqnarray}
Moreover,
\begin{eqnarray}
 Tr \; Cor(X, Y) = \EX  \langle X-\EX(X),  Y-\EX(Y)\rangle.
\end{eqnarray}

\end{remark}

\subsection{Gaussian random variables} \index{Gaussian random


Recall that a random variable taking values in $\R^n$
$$
X:\; \Om \to \R^n
$$
is called Gaussian, if for any $a=(a_1, \cdots, a_n)  \in \R^n$,
$X \cdot a = a_1X_1 \cdots + a_nX_n$ is a scalar Gaussian random
variable. A Gaussian random variable  in $\R^n$ is denoted as $X
\sim \N(m, Q)$, with mean vector $m$ and covariance matrix $Q$.
The covariance matrix $Q$ is symmetric and non-negative (i.e.,
eigenvalue $\la_j \geq 0, \; j=1, \cdots, n$). The trace of $Q$ is
written as  $Tr(Q)=\la_1+\cdots + \la_n$. The covariance matrix is
defined as
$$
Q = ( Q_{ij}) = ( \EX [(X_i-m_i)(X_j-m_j)]).
$$
We use the notations  $E(X)=m$ and $ Cov (X)=Q$. The probability
density function for this Gaussian random variable  $X$ in $\R^n$
is
 \begin{eqnarray}
 f(x)=f(x_1, \cdots, x_n)=\frac{\sqrt{\det(A)}}{(2\pi)^{n/2}}
 e^{-\frac12 \sum_{j,k=1}^n (x_j-m_j) a_{jk} (x_k-m_k)},
\end{eqnarray}
where $A=Q^{-1}=(a_{jk})$.

The probability distribution function  of $X$ is
 \begin{eqnarray}
 F(x)= \PX(\om:\; X(\om) \leq x) =\int_{-\infty}^x f(x)dx.
\end{eqnarray}

The probability distribution measure  $\mu$ (or law
$\mathcal{L}_X$) of $X$ is:
\begin{eqnarray}
 \mu (B) = \int_B f(x) dx, \;\;\; B \in \B (\R^n).
\end{eqnarray}

Here are some observations. For $a, b \in \R^n$,

\begin{eqnarray}
\EX \langle X, a\rangle = \EX \sum_{i=1}^n a_i X_i = \sum_{i=1}^n
a_i \EX (X_i) = \sum_{i=1}^n a_i m_i = \langle m, a\rangle
\end{eqnarray}

\begin{eqnarray}
\EX (\langle X-m, a\rangle \langle X-m, b\rangle  )
 &=& \EX ( \sum_i a_i (X_i-m_i) \; \sum_j b_j (X_j-m_j) ) \\
 &=&    \sum_{i,j} a_i b_j \EX [(X_i-m_i)(X_j-m_j)] \\
&=&    \sum_{i,j} a_i b_j Q_{ij} \\
&=&    \langle Q a, b\rangle
\end{eqnarray}
In particular, $ \langle Q a, a\rangle=\EX  \langle X-m,
a\rangle^2 \geq 0$, which confirms that $Q$ is non-negative. Also,
$\langle Qa, b\rangle=\langle a, Qb\rangle$, which implies that
$Q$ is symmetric.

\bigskip

\begin{definition}
\label{Gaussian} \index{Gaussian random variable in Hilbert space}
A random variable $X: \Om \to H$ in Hilbert space $H$ is called a
Gaussian random variable and denoted as $X \sim \N (m, Q)$, if for
every $a$ in $H$, the real random variable $\langle X, a\rangle$
is a scalar Gaussian random variable (i.e., taking values in
$\R^1$).
\end{definition}

\begin{remark}
If $X $ is a Gaussian random variable taking values in Hilbert
space $H$, then for all $a, b\in H$,

(i) Mean vector $\EX (X)=m: \; \EX \langle X, a\rangle = \langle
m, a\rangle;$

(ii) Covariance operator $Cov (X)=Q$: $\EX (\langle X-m, a\rangle
\langle X-m, b\rangle  )=\langle Qa, b\rangle  $
\end{remark}

\begin{remark}
The Borel probability measure $\mu$ on $(H, \B(H))$, induced by a
  Gaussian random variable $X$ taking values in Hilbert space
$H$, is called a Gaussian measure. If $\mu$ is a Gaussian measure
in $H$, then there exist an element $m \in H$ and a non-negative
symmetric continuous linear operator $Q$: $H\to H$ such that: For
all $h, h_1, h_2 \in H$,

(i) Mean vector $m$:  $\int_H \langle h, x\rangle d\mu(x)= \langle
m, h\rangle;$

(ii) Covariance operator $Q$: $\int_H \langle h_1, x\rangle
\langle h_2, x\rangle d\mu(x)- \langle m, h_1\rangle \langle m,
h_2\rangle =\langle Qh_1, h_2\rangle  $
\end{remark}

Since  the covariance operator $Q$ is non-negative and symmetric,
the eigenvalues of $Q$ are non-negative and the eigenvectors
$e_n$'s form an orthonormal basis for Hilbert space $H$:
$$
Q e_n = q_n e_n, \; n=1, 2, \cdots.
$$
Moreover, trace $Tr (Q) = \sum_{n=1}^{\infty} q_n$.

Note that
\begin{eqnarray}
X-m = \sum X_n e_n
\end{eqnarray}
with coefficients $X_n = \langle X-m, e_n\rangle$.
\begin{eqnarray}
\EX X_n^2 = \EX (\langle X-m, e_n\rangle \; \langle X-m, e_n
\rangle ) = \langle Qe_n, e_n\rangle = \langle q_n e_n, e_n\rangle
=q_n.
\end{eqnarray}

Therefore,
\begin{eqnarray}
\|X-m\|^2 = \sum X_n^2
\end{eqnarray}

\begin{eqnarray}
\EX \|X-m\|^2 = \sum \EX X_n^2 =\sum q_n =Tr (Q).
\end{eqnarray}

\medskip


 We use $L^2(\Omega, H)$, or just $L^2(\Omega)$, to denote the (new) Hilbert space of
square-integrable random variables $x:\; \Omega \to H$. In Hilbert
space $L^2(\Omega, H)$, the scalar product is
\[
< x, y> = \EX <x(\om), y(\om)>,
\]
where $\EX$ denotes the mathematical expectation (or mean) with
respect to probability $\mathbb{P}$. This scalar product induces
the usual mean square norm
\[
\|x\|: =\sqrt{\EX  \|x(\om) \|^2},
\]
which provides an appropriate convergence concept.

\subsection{Brownian motion}

Recall that a Brownian motion (or Wiener process) $W(t)$,   also
denoted as $W_t$, in $R^n$, is a Gaussian stochastic process on a
underlying probability space $(\Om, \cF, \PX)$, where $\Om$ is a
sample space, $\cF$ is a $\sigma-$field composed of measurable
subsets of $\Om$ (called ``events"), and $\PX$ is a probability
(also called probability measure). Being a Gaussian process, $W_t$
is characterized by its mean vector (taking to be the zero vector)
and its covariance operator, a $n\times n$ symmetric positive
definite matrix (taking to be the identity matrix). More
specifically, $W_t$ satisfies the following conditions
\cite{Oksendal}:

\noindent (a)\ \ \ W(0)=0, \; a.s.\\
(b)\ \ \ W has continuous paths or trajectories, \; a.s.\\
(c)\ \ \ W has independent increments,  \\
(d)\ \ \ W(t)-W(s) $\sim$ $N(0, (t-s)I)$, $t\  and\  s >0 \ and\ t
\geq s \geq 0$, where $I$ is the $n\times n$ identity matrix. The
Brownian motion in $R^1$ is called a scalar Brownian motion.

\begin{remark}
(i) The covariance operator here is a constant  $n\times n$
identity matrix $I$, i.e., $Q=I$ and $ Tr(Q)=n$.

(ii) $W(t)  \sim N(0, t I)$, i.e., $W(t)$ has probability density
function $p_t(x) = \frac1{(2\pi t)^{\frac{n}{2}}}
e^{-\frac{x_1^2+...+x_n^2}{2t}}$.

(iii) For every $\a \in (0, \frac12)$, for a.e. $\om \in \Om$,
there exists $C(\om)$ such that

$$ |W(t, \om)-W(s, \om)| \leq
C(\om) |t-s|^{\a},
$$
namely, Brownian paths are H\"older continuous with exponent less
than one half.
\end{remark}

Note that the generalized time derivative of Brownian motion $W_t$
is a mathematical model for  white noise \cite{Arnold74}.

Now we define Wiener process, or Brownian motion, in Hilbert space
$U$. We consider  a symmetric nonnegative linear operator $Q $ in
$U$. If the trace $Tr (Q) < + \infty$, we say $Q$ is a trace class
(or nuclear) operator. Then there exists a complete orthonormal
system (eigenfunctions) $\{e_k\}$ in U , and a  (bounded) sequence
of nonnegative real numbers (eigenvalues) $q_k$ such that
\begin{eqnarray*}
Qe_k = q_k e_k \ , \ k = 1,2, \cdots.
\end{eqnarray*}
A   stochastic process $W(t)$, or $W_t$, taking values in $U$ for
$t \geq 0$ , is called a
Wiener process with covariance operator $Q$ if :\\
(a)\ \ \ $W(0)=0$, \; a.s.\\
(b)\ \ \ $W$ has continuous trajectories, \; a.s.\\
(c)\ \ \ $W$ has independent increments,  \\
(d)\ \ \ $ W(t)-W(s)  \sim$ N(0,(t-s)Q),  $  t \geq s.$
\\
Hence, $\EX W(0)=0$ and $ Cov (W(t))= t Q$.

We can think the covariance matrix $Q$ as a $\infty \times \infty$
diagonal matrix, with diagonal elements $q_1. q_2, \cdots, q_n,
\cdots$.

For any $a\in H$,
$$
a = \sum_n <a, e_n> e_n
$$
$$
Qa =\sum_n <a, e_n> Qe_n = \sum_n q_n <a, e_n>  e_n
$$

We define, for $ \gamma>0$, especially for $\gamma \in (0, 1)$,
\begin{eqnarray}
Q^{\gamma}a =  \sum_n q_n^{\gamma} <a, e_n>  e_n,
\end{eqnarray}
when the right hand side is defined.
\bigskip

\textbf{Representations of Brownian motion in Hilbert space: }

It is known that $W_t$ has an infinite series representation
\cite{DaPrato}:
\begin{eqnarray}
  W_t(\om) = \sum_{n=1}^{\infty} \sqrt{q_n} W_n(t) e_n,
\end{eqnarray}
where
\begin{equation}
            W_n(t) :=
            \left\{
            \begin{array}{lll}
                \frac{\langle W(t),\;  e_n \rangle}{\sqrt{q_n}} & , &\; q_n > 0,\\
                0 & , & \; q_n=0.
            \end{array}
            \right.
            \label{wn}
        \end{equation}
are the standard scalar independent Brownian motions. Namely,
$W_n(t) \sim \N(0, t)$, $\EX W_n(t)=0$,  $\EX W_n(t)^2 =t$ and
$\EX W_n(t)W_n(s)=\min(t, s)$.

This infinite series converges in $L^2(\Om)$, as long as
$Tr(Q)=\sum q_n < \infty $.

\begin{remark}
For example in $H= L^2(0, 1)$, we have an orthonormal basis $e_n=
\sin(n\pi x)$ . In the above infinite series representation,
taking derivative with respect to $x$, we get
\begin{eqnarray}
 \p_x W_t(\om) = \sum_{n=1}^{\infty} \sqrt{2} (n\pi) \sqrt{q_n} W_n(t) \cos(n\pi
 x).
\end{eqnarray}
 In order
for this series to converge, we need $ \sqrt{2} (n\pi) \sqrt{q_n}$
  converges to zero sufficiently fast as $n \to \infty$. So $q_n$ being small
helps. In this sense, the trace $Tr(Q)=\sum q_n$  may be seen as a
measurement for spatial regularity of white noise $\dot{W}_t$: the
smaller the trace $Tr(Q)$, the more regular of the noise.
\end{remark}

\bigskip

We do some calculations. For $a, b \in H$, we have the following
identities.
\begin{eqnarray}
 \EX \langle W_t, W_t \rangle = \EX \|W_t\|^2
&=& \EX \langle \sum_{n=1}^{\infty} \sqrt{q_n} W_n(t) e_n,
\sum_{n=1}^{\infty} \sqrt{q_n} W_n(t) e_n \rangle  \nonumber \\
&=&  \sum_{n=1}^{\infty} q_n \EX \langle W_n(t), W_n(t)\rangle  \nonumber \\
&=& t \sum_{n=1}^{\infty} q_n = t \; Tr (Q).
\end{eqnarray}

\begin{eqnarray}
 \EX \langle W_t, a \rangle &=&   \langle 0, a \rangle =0.
\end{eqnarray}

\begin{eqnarray*}
 \EX (\langle W_t, a\rangle \; \langle W_t, b\rangle)
&=& \EX [\langle \sum_{n=1}^{\infty} \sqrt{q_n} W_n(t) e_n, a
\rangle \; \langle \sum_{n=1}^{\infty} \sqrt{q_n} W_n(t) e_n, b
\rangle  ] \\
&=& \EX  \sum_{m,n} \sqrt{q_m q_n} W_m(t)W_n(t) <e_m, a><e_n, b> \\
&=& \sum_{n} t q_n  <e_n, a><e_n, b> \\
&=& t \sum_{n}    < e_n, a><q_n e_n, b> \\
&=& t \sum_{n}    < e_n, a><Q e_n, b> \\
&=& t \sum_{n}     <Q <e_n, a>e_n, b> \\
&=& t  <Q \sum_{n} <e_n, a>e_n, b> \\
&=& t  <Q a, b>,
\end{eqnarray*}
where we have used the fact that $a=\sum_{n} <e_n, a>e_n$ in the
final step.

In particular, taking $a=b$, we obtain
\begin{eqnarray}
 \EX \langle W_t, a \rangle^2 &=&  t \langle Qa, a\rangle, \\
 Var ( \langle W_t, a \rangle ) &=&  t \langle Qa, a\rangle.
\end{eqnarray}

More generally,
\begin{eqnarray}
\EX (\langle W_t, a\rangle \; \langle W_s, b\rangle)= \min(t,s)<Q
a, b>.
\end{eqnarray}


Moreover,
\begin{eqnarray}
 \EX [ W_t(x) W_s (y) ]
&=& \EX \{ \sum_{n=1}^{\infty} \sqrt{q_n} W_n(t) e_n(x)
\sum_{m=1}^{\infty} \sqrt{q_m} W_m(s) e_m(y) \}  \nonumber \\
&=& \sum_{n,
m=1}^{\infty} \sqrt{q_n q_m} \; \EX [W_n(t)W_m(s)] e_n(x) e_m(y) \nonumber \\
&=& \min(t, s)  \sum_{n=1}^{\infty} q_n  e_n(x)e_n(y)  \nonumber  \\
&=&  \min(t,s) q(x,y),
\end{eqnarray}
where
$$ q(x,y) =
\sum_{n=1}^{\infty} q_n  e_n(x)e_n(y).
$$

On the other hand, the covariance operator may be represented in
terms of $q(x,y)$:
\begin{eqnarray}
Qa = Q \sum_{n} <e_n, a>e_n &=&   \sum_{n} <e_n, a> Q e_n   \nonumber  \\
& =& \sum_{n} <e_n, a> q_n e_n  \nonumber\\
& =& \sum_{n} \int_0^1 a(y)e_n(y)dy q_n e_n(x) \nonumber \\
&=& \int_0^1 q(x,y)a(y)dy.
\end{eqnarray}

Sometimes we call the kernel function $q(x,y)$ the spatial
correlation. The smoothness of $q(x,y)$ depends on the decaying
property of $q_n$'s.

\section{Stochastic Partial Differential Equations}
\label{spde}

\subsection{Stochastic calculus in Hilbert space}

We define the Ito stochastic integral:
\begin{eqnarray*}
 \int_0^T \Phi(s, \om) dW_s.
\end{eqnarray*}
 Note that since $W_t$ takes values in Hilbert space $U$.   The integrand
   $ \Phi(t, \om)$ is usually a linear operator from $U$ to $H$
(for each time $t$ and each sample $\om$):
 $$
\Phi:  U \to H.
$$
It is also possible to take $W_t$ as a scalar, real-valued
Brownian motion. For example, in $\int_0^T u(s)dW_s$, if $W_t$ is
a scalar Brownian motion, we can interpret the integrand $u$ as a
multiplication operator.

For Brownian motion $W_t$ in $U$
\begin{eqnarray}
  W_t(\om) = \sum_{n=1}^{\infty} \sqrt{q_n} W_n(t) e_n,
\end{eqnarray}
we define
\begin{eqnarray}
 \int_0^T \Phi(s, \om) dW_s(\om)
 = \sum_{n=1}^{\infty} \sqrt{q_n} \int_0^T \Phi(s, \om) e_n\; dW_n(s).
\end{eqnarray}

 A property of Ito integrals:

\begin{eqnarray}
 \EX \int_0^T \Phi(s, \om) dW_s(\om) =0.
\end{eqnarray}

\subsection{Deterministic calculus in Hilbert space}

In order to discuss more tools to handle stochastic calculus in
Hilbert space, we need to recall some concepts of deterministic
calculus.

 For calculus in Euclidean space $\R^n$, we have concepts
\emph{derivative} and \emph{directional derivative}.
\index{directional derivative} In Hilbert space, we have the
corresponding \emph{Fr\'echet derivative} and \emph{Gateaux
derivative} \cite{Berger, Zeidler2}.

\index{Fr\'echet derivative}

\index{Gateaux derivative}

 Let $H$ and $\hat{H}$ be two Hilbert spaces, and $F:
U \subset H \to \hat{H}$ be a map, whose domain of definition $U$
is an open subset of $H$. Let $L(H, \hat{H})$ be the set of all
bounded linear operators $A: H \to \hat{H}$.  In particular, $L(H)
:= L(H, H)$. We can also introduce a multilinear operator $A_1: H
\times H \to \hat{H}$. The space of all these multilinear
operators is denoted as $L(H \times H, \hat{H})$.

\begin{definition}
The map $F$ is Fr\'echet differentiable at $u_0 \in U$ if there is
a linear bounded operator $A: H \to \hat{H}$ such that
$$
\lim_{h\to 0} \frac{\| F(u_0+h) -F(u_0)-A h\|}{\| h\|} =0, \\
i.e., \| F(u_0+h) -F(u_0)-A h\| = o (\|h\|),
$$
where $\| \cdot \|$ denotes norms in $H$ or $\hat{H}$ as
appropriate. The linear bounded operator  $A$ is called the
Fr\'echet derivative of  $F$ at $u_0$, and is denoted as
$F_u(u_0)$, or sometimes $F'(u_0)$.
\end{definition}

If $F$ is linear, its Fr\'echet derivative is itself.

\begin{definition}
 The directional derivative of $F$ at $u_0 \in U$ in the direction
 $h \in H$ is defined by the limit
 $$
 \delta F(u_0; h) := \lim_{t\to 0}\frac{ F(u_0+t h) -F(u_0)}{t}.
 $$
 If this limit exists for every $h \in H$, and
 $F'_G(u_0) h:= \delta F(u_0; h)$ is a linear map, then we say
 that $F$ is Gateaux differentiable at $u_0$. The linear map
 $F'_G(u_0)$ is called the Gateaux derivative of $F$ at $u_0$.
\end{definition}

In fact,  if $F$ is Fr\'echet differentiable at $u_0$, then it is
also Gateaux differentiable at $u_0$ and they are equal
\cite{Berger, Zeidler2}:
$$
F_u(u_0) = F'_G(u_0).
$$
But the converse is not usually true. It is true under suitable
conditions; see \cite{Berger}, p. 68.

\bigskip

For any nonlinear  map $F:  U \subset H  \to Y$, its Fr\'echet
derivative  $F'(u_0)$ is a linear operator, i.e., $F'(u_0) \in
L(H, Y)$. Similarly, we can define higher order Fr\'echet
derivatives. Each of these derivatives is a multilinear operator.
For example,
\begin{eqnarray*}
f''(u_0) &:&   H \times H \to Y, \\
        & &  (h, k) \rightarrowtail f''(u_0)(h, k).
\end{eqnarray*}
We denote
\begin{eqnarray*}
f''(u_0)h^2 &:=&   f''(u_0)(h, h), \\
f^{\prime \prime \prime}(u_0)h^3  &:=& f^{\prime \prime
\prime}(u_0)(h,h,h),
\end{eqnarray*}
and similarly for higher order derivatives.

 Then we have the Taylor expansion in Hilbert space
\begin{eqnarray*}
 f(u+ h) =  f(u)+ f'(u)h + \frac1{2!} f''(u)h^2
 +\cdots + \frac1{m!} f^{(m)}(u)h^m + R_{m+1}(u, h),
\end{eqnarray*}
where the remainder
\begin{eqnarray*}
R_{m+1}(u, h) = \frac1{(m+1)!} \int_0^1 (1-s)^m f^{(m+1)}(u+s
h)h^{m+1} ds.
\end{eqnarray*}

\bigskip

\begin{remark}
It is interesting to relate these two concepts with the classical
concept of variational derivative (or functional derivative) that
is used in the context of calculus of variations. The variational
derivative is usually considered for functionals defined as
spatial integrals, such as a Langrange  functional in mechanics.
For example,
$$
F(u) = \int_0^l G(u(x), u_x(x)) dx,
$$
where $u$ is defined on $x\in [0, l]$ and satisfies zero Dirichlet
boundary condition at $x=0, l$. Then it is known \cite{Hunter}
that
\begin{eqnarray}  \label{varation}
F_u(u) h = \int_0^l \frac{\delta F}{\delta u} h(x) dx,
\end{eqnarray}
for $h$ in the Hilbert space $H^1_0(0, l)$. The quantity
$\frac{\delta F}{\delta u}$ is the classical variational
derivative of $F$. The equation (\ref{varation}) above gives the
relation between Fr\'echet derivative and variational derivative.
\end{remark}

\subsection{Ito's formula in Hilbert space}

We get back to stochastic calculus in Hilbert space $H$. We first
look at the Ito's formula; see \cite{DaPrato} or \cite{Rockner}.

\begin{theorem}
Let $u$ be the solution of the  SPDE
\begin{eqnarray}
 du = b(u) dt + \Phi(u) dW_t, \; \; u(0)=u_0.
\end{eqnarray}
Assume that $F(t,u)$ be a given smooth (deterministic) function:
$$
F: [0, \infty) \times H \to \R^1.
$$

Then

(i) Ito's Formula: Differential form
\begin{eqnarray}
dF(t, u(t)) &=&
  F_u(t, u(t)) ( \Phi(u(t))dW_t )  \nonumber \\
& + & \{ F_t(t, u(t))+   F_u(t, u(t)) (b(u(t))  \nonumber \\
& + & \frac12 Tr [F_{uu}(t, u(t)) (\Phi(u(t))Q^{\frac12}) \;
(\Phi(u(t))Q^{\frac12})^* ] \} dt,
\end{eqnarray}
where $F_u$ and $F_{uu}$ are   Fr\'echet derivatives, $F_t$ is the
usual partial derivative in time, and $*$ denotes adjoint
operator. This formula is understood with the following symbolic
operations in mind:
$$
\langle dt, dW_t \rangle = \langle dt, dW_t \rangle =0,\; \\
\langle dW_t, dW_t \rangle = Tr(Q) dt
$$

(ii) Ito's Formula: Integral form
\begin{eqnarray}
 F(t, u(t)) &=& F(0, u(0)) +
\int_0^t   F_u(s, u(s)) ( \Phi(u(s))dW_s)  \nonumber \\
& + & \int_0^t \{ F_t(s, u(s))+   F_u(s, u(s))(b(u(s))  \nonumber \\
& + & \frac12 Tr [F_{uu}(s, u(s)) (\Phi(u(s))Q^{\frac12}) \;
(\Phi(u(s))Q^{\frac12})^* ] \} ds,
\end{eqnarray}
where $F_u$ and $F_{uu}$ are   Fr\'echet derivatives, and $F_t$ is
the usual partial derivative in time. Moreover,
$$
\int_0^t   F_u(s, u(s))(\Phi(u(s))dW_s ) = \int_0^t
\tilde{\Phi}(u(s)) dW_s
$$
and for all $s$, $v \in H$, $\om \in \Om$, the operator
$\tilde{\Phi}(u(s))$ is defined by
$$
\tilde{\Phi}(u(s))(v) :=   F_u(s, u(s))(\Phi(u(s))v ).
$$
Also,
$$
Tr [F_{uu}(s, u(s)) (\Phi(u(s))Q^{\frac12}) \;
(\Phi(u(s))Q^{\frac12})^* ] = Tr [(\Phi(u(s))Q^{\frac12})^* \;
F_{uu}(s, u(s)) \; (\Phi(u(s))Q^{\frac12})].
$$
\end{theorem}

 Note that for the symmetric non-negative covariance
operator $Q$ with eigenvalues $q_n\geq 0$ and eigenvector $e_n$,
$n=1,2, \cdots$, we have
$$
Q u = \sum_n q_n \langle u, e_n\rangle e_n, \\
Q^{\frac12}u = \sum_n q_n^{\frac12} \langle u, e_n\rangle e_n.
$$
In fact, for a given function $h: \R \to \R$, we define the
operator $h(Q)$ through the following natural formula
(\cite{Zeidler1}, p. 293-294),
$$
h(A)u = \sum_n h(q_n)  \langle u, e_n\rangle e_n,
$$
when the right hand side is defined.

\bigskip

\begin{example}
A typical application of Ito's formula for SPDEs.

\begin{eqnarray}
 du = b(u) dt + \Phi(u) dW_t, \; \; u(0)=u_0.
\end{eqnarray}
Take Hilbert space $H=L^2(D), D \subset \R^n$, with the usual
scalar product $\langle u, v\rangle = \int_D u v dx$.

Energy functional $F(u)= \frac12 \int_D u^2 dx=\frac12 \|u\|^2$.
In this case, $F_u(u) (h) = \int_D u h dx$ and $F_{uu}(u) (h,
k)=\int_D h(x) k(x)dx$.

\begin{eqnarray*}
 \frac12 d \|u\|^2
 & = & \{\langle u, b(u)\rangle
 + \frac12 Tr\int_D  [(\Phi(u)Q^{\frac12}) (\Phi(u)Q^{\frac12})^*]dx \} dt  +
 \langle u, \Phi(u) dW_t \rangle,
\end{eqnarray*}

Integrating and taking mathematical expectation, we obtain
%

\begin{eqnarray*}
 \frac12  \EX \|u\|^2
 & = &  \frac12  \EX \|u(0)\|^2+ \EX \int_0^t \langle u, b(u)  \rangle dt
 + \frac12 \EX \int_0^t Tr\int_D [(\Phi(u(r))Q^{\frac12})
 (\Phi(u(r))Q^{\frac12})^*]dx
 dr
\end{eqnarray*}

Note that in this special case, $F_u$ is a bounded operator in
$L(H, \R)$, which can be identified with $H$ itself due to the
Riesz representation theorem.
\end{example}

\begin{example}

Energy functional $F(u)=   \int_D |u|^{2p} dx =   \int_D (|u|^2)^p
dx$ for $p \in [1, \infty)$. In this case, $F_u(u_0) (h) = 2p
\int_D |u_0|^{2p-2}u_0  h dx$ and $F_{uu}(u_0) (h, k)=2p \int_D
|u_0|^{2p-2} h k dx + 4p(p-1)\int_D |u_0|^{2p-4} u_0 h(x) u_0
k(x)dx$.

\end{example}

\begin{example}    (\cite{Chow}, p. 153)
Let $H$ be a Hilbert space with scalar product $ <\cdot, \cdot>$
and norm $\| \cdot \|^2 = <\cdot, \cdot>$.

Consider an energy functional $F(u)= \|u\|^{2p}  $ for $p \in [1,
\infty)$. In this case, $F_u(u_0) (h) = 2p \|u_0\|^{2p-2} <u_0,
h>$ and
\begin{eqnarray*}
F_{uu}(u_0) (h, k)=2p \|u_0\|^{2p-2} <h, k>   + 4p(p-1)
\|u_0\|^{2p-4} <u_0, h> \; <u_0, k>  \\
=2p \|u_0\|^{2p-2} <h, k>   + 4p(p-1) \|u_0\|^{2p-4} <(u_0\otimes
u_0)h, k>,
\end{eqnarray*}
where $(a\otimes b) h :=a<b, h>$; see \S \ref{tools-Hilbert} or
\cite{DaPrato}, p25.
\end{example}

\textbf{Stochastic product rule}:

Let $u$ and $v$ be solutions of two SPDEs. Then
\begin{eqnarray}
 d (u v )= u  dv  + (du ) v  + du  dv.
\end{eqnarray}

\textbf{Ito isometry}:

\begin{eqnarray}
\EX \|\int_0^t \Phi(t, \om) dW_t \|^2 =\EX \int_0^t
Tr[(\Phi(r)Q^{\frac12}) (\Phi(r)Q^{\frac12})^*] dr.
\end{eqnarray}

\textbf{Generalized Ito isometry}:

\begin{eqnarray}
 & &\EX  \langle\int_0^a F(t,\omega) dW_t, \int_0^b G(t,\omega)
 dW_t\rangle    \nonumber \\
 &=&  \EX \int_0^{a \wedge b}  Tr [ (G(r, \om)Q^{\frac12}) (F(r,
 \om)Q^{\frac12})^*] dr,
\end{eqnarray}
where $a \wedge b = \min(a, b)$.

\subsection{Stochastic partial differential equations}

A general class of SPDEs may be written  as
\begin{eqnarray}
 du_t = [Au + f(u)] dt + G(u) dW_t,
\end{eqnarray}
where $Au$ is the linear part, $f(u)$ is the nonlinear part,
$G(u)$ the noise intensity (usually an   operator),  and $W_t$ a
Brownian motion. When $G$ depends on $u$, $  G(u) dW_t$ is called
a multiplicative noise, otherwise it is an additive noise.

  For general background on SPDEs, such as wellposedness and basic
  properties of solutions, see \cite{DaPrato, Rockner, Roz}.

\section{Correlation}
\label{corr}

In this section, we discuss correlation of solutions, at different
time instants,  of some linear SPDEs. We first recall some
information about Fourier series in Hilbert space.

\subsection{Hilbert-Schmidt theory and Fourier series in Hilbert
space}

A separable Hilbert space $H$ has a countable orthonormal basis
$\{ e_n\}_{n=1}^{\infty}$. Namely, $\langle e_m,
e_n\rangle=\delta_{mn}$, where $\delta_{mn}$ is the Kronecker
delta function. Moreover,for any $h \in H$, we have Fourier series
expansion
\begin{eqnarray}
 h = \sum_{n=1}^{\infty} \langle h, e_n\rangle e_n.
\end{eqnarray}

In the context of solving stochastic PDEs, we may chose to work on
a Hilbert space with an appropriate orthonormal basis. This is
naturally possible with the help of the Hilbert-Schmidt theory
\cite{Zeidler1}, p.232.

The Hilbert-Schmidt theorem (\cite{Zeidler1}, p.232) says that a
linear compact symmetric operator $A$ on a separable Hilbert space
$H$ has a set of eigenvectors that form a complete orthonormal
basis for $H$. Moreover, all the eigenvalues of $A$ are real, each
non-zero eigenvalue has finite multiplicity, and two eigenvectors
that correspond to different eigenvalues are orthogonal.

This theorem applies to a strong (self-adjoint)  elliptic
differential operator  $B$
$$
Bu = \sum_{0 \leq |\a|,|\b|\leq m} (-1)^{|\a|} D^{\a} (a_{\a\b}(x)
D^{\b}u), \; \; x\in D \subset \R^n,
$$
where the domain of definition of $B$ is an appropriate dense
subspace of $H=L^2(D)$, depending on the boundary condition
specified for $u(x)$.

 In this
case, $A:=B^{-1}$ is a linear symmetric compact operator in a
Hilbert space, e.g., $H=L^2(D)$. We may consider $(L+ a I)^{-1}$.
This may be necessary in order for the operator to be invertible,
i.e., no zero eigenvalue, such as in the case of Laplace operator
with zero Neumann boundary condition.

By the Hilbert-Schmidt theorem, eigenvectors (also called
eigenfunctions in this context) of $A =B^{-1}$ form an orthonormal
basis for $H=L^2(D)$. Note that $A$ and $B$ share the same set of
eigenfunctions. So we can claim that the strong elliptic operator
$B$'s eigenfunctions form an orthonormal basis for $H=L^2(D)$.

 In the case of one spatial variable, the elliptic differential operator
is the so called Sturm-Liouville operator \cite{Zeidler1}, p.245.
For example
$$
Bu = -(pu')' + qu,  \; x \in (0, l),
$$
where $p(x)$, $p'(x)$ and $q(x)$ are continuous on $(0, l)$. This
operator arises in the method of separating variables for solving
linear (deterministic) partial differential equations in the next
section. By the Hilbert-Schmidt theorem, eigenfunctions of the
 Sturm-Liouville operator  form an orthonormal basis for $H=L^2(0, l)$.

\subsection{The wave equation with additive noise}

Consider the stochastic wave equation with additive noise:
\begin{eqnarray}
 u_{tt} = c^2 u_{xx}  + \e W_t,\;  0<x<l,\; t>0 \\
 u(0, t)=u(l, t)=0, \\
 u(x,0) = f(x), \;\; u_t(x, 0)=g(x),
\end{eqnarray}
where   $\e$ is a real parameter modeling the noise intensity,
$c>0$ is a constant (wave speed), and $W_t$ is a   Brownian motion
taking values in Hilbert space $H=L^2(0, l)$.

Method of eigenfunction expansion:
\begin{eqnarray}
u = \sum_{n=1}^{\infty} u_n(t) e_n(x),  \label{soln-wave} \\
W_t =\sum_{n=1}^{\infty} \sqrt{q_n} W_n(t)  e_n(x),
\end{eqnarray}
where
$$
e_n(x) = \sqrt{2/l} \; \sin\frac{n\pi x}{l}, \; \la_n =
 (n\pi/l)^2, n=1, 2, \cdots.
$$

Putting these into the SPDE  $u_{tt} = c^2 u_{xx}  + \e W_t$we
obtain
\begin{eqnarray}
  \ddot{u}_n(t) +c^2 \lambda_n u_n = \e \sqrt{q_n} \dot{W}_n(t),
  n=0, 1, 2, \cdots.
 \end{eqnarray}
For each $n$,  this second order SDE may be   solved  by
converting to a linear system of first order SDEs \cite{Oksendal}:
\begin{eqnarray}
  u_n(t)&=& [A_n-\e \frac{l}{cn\pi}
\sqrt{q_n}\int_0^t \sin\frac{cn\pi}{l}s dW_n(s) ] \;
\cos\frac{cn\pi}{l}t  \nonumber \\
& + &  [B_n+\e \frac{l}{cn\pi} \sqrt{q_n}\int_0^t
\cos\frac{cn\pi}{l}s dW_n(s) ] \; \sin\frac{cn\pi}{l}t
 \end{eqnarray}
 with  $A_n$ and $ B_n$ constants.

 The final solution is
\begin{eqnarray}
u(x,t) &=& \sum_{n=1}^{\infty} \{ [A_n-\e \frac{l}{cn\pi}
\sqrt{q_n}\int_0^t \sin\frac{cn\pi}{l}s dW_n(s)   ]
\cos\frac{cn\pi}{l}t  \nonumber \\
& + &  [B_n+\e \frac{l}{cn\pi} \sqrt{q_n}\int_0^t
\cos\frac{cn\pi}{l}s dW_n(s)  ] \sin\frac{cn\pi}{l}t \}e_n(x),
\end{eqnarray}
where the constants $A_n$ and $ B_n$ are determined by the initial
condition as follows
$$
A_n = \langle f, e_n\rangle, \;\;\; B_n =\frac{l}{cn\pi} \langle
g, e_n\rangle.
$$

When the noise is at one mode, say at the first mode $e_1(x)$
(i.e., $q_1>0$ but $q_n=0,\; n=2, 3, \cdots.$), we see that the
solution contains randomness only at that mode. So for the linear
stochastic diffusion system, there is no interactions between
modes. In other words, if we randomly force a few fast modes, then
there is no   impact on slow modes.

\bigskip

\textbf{Mean value} for the solution:

\begin{eqnarray}
 \EX u(x,t)=\sum^{\infty}_{n=1}[A_n\cos({cn\pi t \over l})
 +B_n\sin({cn\pi t\over l})]e_n(x).
\end{eqnarray}

\textbf{Covariance}  for the solution:

Now we calculate the covariance of solution $u$ at different time
instants $t$ and $s$, i.e., $\EX <u(x,t)-\EX u(x,t), u(x,s)-\EX
u(x,s)>$.



Using the Ito's isometry, we   get
\begin{eqnarray*}
 & & \EX <u(x,t)-\EX u(x,t), u(x,s)-\EX u(x,s)>  \\
 &=&  \sum^{\infty}_{n=1}{\epsilon^2
l^2q_n\over c^2n^2 {\pi}^2} \; [\int_0^{t\wedge s}\sin^2{cn\pi r\over l} dr \cos {cn\pi t\over l}\cos {cn\pi s\over l} \\
&+& \int^{t\wedge s}_0 \cos^2{cn\pi r\over l}dr \sin{cn\pi
t\over l} \sin{cn\pi s\over l} \\
&-& \int^{t\wedge s}_0 \sin{cn\pi r\over l} \cos {cn\pi r\over
l}dr ( \cos {cn\pi t\over l}\sin {cn\pi s\over l}+ \cos {cn\pi
s\over l}\sin {cn\pi t\over l}) ]
\end{eqnarray*}
After integrations,   we   get the covariance as
\begin{eqnarray*}
Cov(u(x,t), u(x,s))
&=& \EX <u(x,t)-\EX u(x,t), u(x,s)-\EX u(x,s)> \\
&=& \sum^{\infty}_{n=1}{\epsilon ^2 l^2 q_n\over
 2c^2n^2\pi^2}[({t\wedge s}) \cos{cn\pi (t-s)\over l}  \\
&-& {l\over 2cn\pi} \; \sin{2cn\pi ({t\wedge s}) \over l} \cos
 { cn\pi (t+s)\over l}     \\
&+& {l\over 2cn\pi}\; \cos{2cn\pi ({t\wedge s}) \over l} \sin
 { cn\pi (t+s)\over l}     \\
 &-& {l\over 2cn\pi} \sin{cn\pi (t+s)\over l}] \\
&=&  \sum^{\infty}_{n=1}{\epsilon ^2 l^2 q_n\over
 2c^2n^2\pi^2}[({t\wedge s}) \cos{cn\pi (t-s)\over l}  \\
&+& {l\over 2cn\pi} \; \sin{ cn\pi (t+s-2 ({t\wedge s})) \over l}   \\
 &-& {l\over 2cn\pi} \sin{cn\pi (t+s)\over l}] .
\end{eqnarray*}
In particular, for $t=s$ we   get the variance.

\bigskip

\textbf{Variance} for the solution:

\begin{equation}
 Var (u(x,t)) =\sum^{\infty}_{n=1}{\epsilon^2 l^2q_n\over
2c^2 n^2 \pi^2} [  t-\frac{l}{2cn\pi}\sin (\frac{2cn\pi}{l}t)].
\end{equation}

\bigskip

\textbf{Energy evolution} for the solution:
\begin{equation}
  E(t) = \frac12 \int_0^l [u_t^2 + c^2 u_x^2] dx
\end{equation}
Taking time derivative,
\begin{equation}
  \dot{E}(t) =  \int_0^l  u_t [u_{tt} - c^2 u_{xx}] dx
  = \e \int_0^l u_t(x,t) \dot{W}_t(x) dx.
\end{equation}
Or in integral form,
\begin{eqnarray*}
E(t)&=&E(0) + \e \int_0^t \int_0^l u_s(x,t) dW_s(x) dx\\
&=& E(0) + \e \int_0^l \int_0^t \p_t u(x,s) dW_s(x) dx
\end{eqnarray*}
Thus
\begin{eqnarray}
\EX E(t)=E(0).
\end{eqnarray}

\begin{equation}
 Var(E(t))=\varepsilon^2\; \EX\;(\int_0^l\int_0^t \p_t u(x, s)  dW_s dx)^2,
\end{equation}
where   $W_t$ is in the following form
\begin{eqnarray}
 W_t=W(t) &=&\sum_{n=1}^{\infty} \sqrt{q_n}W_n(t)e_n(x),
\end{eqnarray}
 and   $\p_t u$ can be written as
\begin{eqnarray}
 \p_t u
&=& \sum \{ -A_n \frac{ c n \pi}{l}\sin \frac{ c n \pi t}{l}
 +B_n\frac{ c n\pi }{l}\cos \frac{ c n\pi t}{l}   \nonumber \\
&+& \varepsilon\sqrt{q_n} \; [\int_0^t\sin \frac{ c n\pi s}{l}
\; dW_n(s) \; \sin \frac{c n\pi t}{l}   \nonumber \\
&+& \int_0^t \cos \frac{c n\pi s}{l} dW_n(s) \; \cos \frac{c n \pi
t}{l} ] \} \; e_n(x).
\end{eqnarray}
Set $\frac{cn\pi}{l}=\mu_n$ and rewrite
\begin{eqnarray*}
\partial_tu&=&\sum\{F_n(t)+\varepsilon\sqrt{q_n}[\int_0^t(\sin\mu_ns\cdot\sin\mu_nt+\cos\mu_ns\cdot\cos\mu_nt)dW_n(s)]\}e_n(x)\\
&=&\sum\{F_n(t)+\varepsilon\sqrt{q_n}\int_0^t\cos\mu_n(t-s)dW_n(s)\}e_n(x),
\end{eqnarray*}
where
$$F_n(t):=-A_n\mu_n\sin\mu_nt+B_n\mu_n\cos\mu_nt, \;n=1, 2, \cdots$$
 For the simplicity of notations, set
 $$G_n(t):=F_n(t)+\varepsilon\sqrt{q_n}\int_0^t\cos\mu_n(t-s)dW_n(s), \;n=1, 2, \cdots$$
 then we have
 $$\partial_tu=\sum G_n(t)e_n(x).$$
 Thus
 \begin{eqnarray*}
 \mathbb{E}(\int_0^l\int_0^t \p_t u(x, s)dW_sdx)^2
&=&\mathbb{E}[\int_0^l(\sum\limits_{n=1}^\infty
\sqrt{q_n}e_n(x)\int_0^tu_sdW_n(s))dx]^2\\
&=&\mathbb{E}[\sum\limits_{n=1}^\infty\sqrt{q_n}\int_0^l\int_0^tu_se_n(x)dW_n(s)dx]^2\\
&=&\mathbb{E}[\sum\limits_{n=1}^\infty\sqrt{q_n}\int_0^t(\int_0^le_n(x)\sum\limits_{j=1}^\infty
G_j(s)e_j(x)dx)dW_n(s)]^2\\
&=&\mathbb{E}[\sum\limits_{n=1}^\infty\sqrt{q_n}\int_0^t(\sum\limits_{j=1}^\infty
G_j(s)\int_0^le_n(x)e_j(x)dx)dW_n(s)]^2\\
&=&\mathbb{E}[\sum\limits_{n=1}^\infty\sqrt{q_n}\int_0^tG_n(s)dW_n(s)]^2\\
\end{eqnarray*}
\begin{eqnarray*}
&=&\sum\limits_{n=1}^\infty q_n \; \mathbb{E}\int_0^tG_n^2(s)ds\\
&=& \sum\limits_{n=1}^\infty
q_n \; \mathbb{E} \int_0^t[F_n(s)+\varepsilon\sqrt{q_n}\int_0^s\cos\mu_n(s-r)dW_n(r)]^2ds\\
&=&\sum\limits_{n=1}^\infty q_n \int_0^t
F_n^2(s)ds+ \mathbb{E}\sum\limits_{n=1}^\infty\varepsilon^2q_n^2\int_0^t[\int_0^s\cos\mu_n(s-r)dW_n(r)]^2ds\\
&=&\sum\limits_{n=1}^\infty q_n\int_0^t
F_n^2(s)ds+\sum\limits_{n=1}^\infty\varepsilon^2q_n^2\int_0^t[\int_0^s\cos^2\mu_n(s-r)dr]ds\\
&=&\sum\limits_{n=1}^\infty
q_n[A_n^2\mu_n^2(\frac{t}{2}-\frac{1}{4\mu_n}\sin2\mu_nt)+B_n^2\mu_n^2(\frac{t}{2}+\frac{1}{4\mu_n}\sin2\mu_nt)\\
&-&\frac{1}{2}A_nB_n\mu_n(1-\cos2\mu_nt)]
 + \sum\limits_{n=1}^\infty\varepsilon^2q_n^2[\frac{t^2}{4}+\frac{1}{8\mu_n^2}(1-\cos2\mu_nt)].
\end{eqnarray*}
Therefore,
\begin{eqnarray*}
Var(E(t))&=& \sum\limits_{n=1}^\infty\varepsilon^2
q_n[A_n^2\mu_n^2(\frac{t}{2}-\frac{1}{4\mu_n}\sin2\mu_nt)+B_n^2\mu_n^2(\frac{t}{2}+\frac{1}{4\mu_n}\sin2\mu_nt)\\
&-&\frac{1}{2}A_nB_n\mu_n(1-\cos2\mu_nt)]
  +
  \sum\limits_{n=1}^\infty\varepsilon^4q_n^2[\frac{t^2}{4}+\frac{1}{8\mu_n^2}(1-\cos2\mu_nt)].
\end{eqnarray*}


\subsection{The   diffusion equation with
multiplicative noise}

Consider   the stochastic Diffusion equations with zero Dirichlet
boundary condition
\begin{eqnarray}
 u_t&=& u_{xx}+\epsilon u \dot{w}_t, \; 0<x<1,  \label{diffu}\\
 u(x, 0)&=&f(x),
\end{eqnarray}
where $w_t$ is a scalar Brownian motion. We take Hilbert space
$H=L^2(0, 1)$ with an orthonormal basis ${e_n=\sqrt{2}\sin(n\pi
x)}$.  We use the method of eigenfunction expansion:
\begin{equation}
u(x,t)=\sum u_n(t)e_n(x),
\end{equation}
\begin{equation}
 u_{xx}=\sum u_n(t)\ddot{e}_n(x) =\sum -u_n(t) (n\pi)^2 e_n(x).
\end{equation}
 Putting these into the above SPDE \eqref{diffu}, with $\lambda_n= (n\pi)^2$,
 we   get
\begin{equation}
\sum\dot{u}_n(t) e_n(x)=\sum u_n(t)(-\lambda_n)e_n+\epsilon \sum
u_n(t)e_n(x)\dot{w}_t.
\end{equation}
We further obtain the following system of SODEs:
\begin{equation}
 du_n(t)=-\lambda_n u_n(t)+\epsilon u_n(t)dw(t), \; n=1, 2, 3,
 \cdots.
\end{equation}
Thus
\begin{equation}
u_n(t)=u_n(0) \exp((-\lambda_n-\frac12\epsilon^2)t+\epsilon w(t)),
\end{equation}
where $u_0=\Sigma <u_0,e_n(x)>e_n(x)$. Therefore, the final
solution is:
\begin{equation}
u(x,t)=\sum a_n e_n(x) \exp(b_n t+\epsilon w_t),
\end{equation}
with $a_n=<f(x),e_n(x)>$ and $b_n=(-\lambda_n- \frac12
 \epsilon^2)$.

Note that $\EX \exp(b_n t+\epsilon w_t)=\exp(b_n t) \EX
\exp(\epsilon w_t) = \exp(b_n t) \exp(\frac12 \epsilon^2 t)=\exp
(-\lambda_n t)$. Therefore, we can find out the mean, variance,
covariance and correlation
 of the solution:
\begin{eqnarray}
 E(u(x,t))=\sum a_n e_n(x) \exp(-\lambda_n t).
\end{eqnarray}

\begin{eqnarray}
 Var(u(x,t))=\EX \langle u(x,t)-E(u(x,t)),u(x,t)-E(u(x,t)) \rangle
 \nonumber \\
 =\sum  a_n^2 \exp(-2\lambda_n t ) [\exp(\epsilon^2 t)-1].
\end{eqnarray}

For $\tau\leq t$, we have
\begin{eqnarray*}
\mathbb{E}\exp\{\epsilon(w_t+w_\tau)\}&=&\mathbb{E}\exp\{\epsilon(w_t-w_\tau)+2\epsilon
w_\tau\}\\
&=&\mathbb{E}\exp\{\epsilon(w_t-w_\tau)\}\cdot\mathbb{E}\exp\{2\epsilon
w_\tau\}\\
&=&\exp\{\frac{1}{2}\epsilon^2(t-\tau)\}\cdot\exp\{2\epsilon^2\tau\}\\
&=&\exp\{\frac{1}{2}\epsilon^2[(t+\tau)+2(t\wedge\tau)]\}.
\end{eqnarray*}
therefore, by direct calculation, we can get
\begin{eqnarray*}
Cov(u(x,t), u(x,\tau))&=&\sum
a_n^2\{\exp(b_n(t+\tau)+\frac{1}{2}\epsilon^2((t+\tau)+2(t\wedge\tau)))+\exp(-\lambda_n(t+\tau))\\
&-&\exp(-\lambda_n\tau+b_nt+\frac{1}{2}\epsilon^2t)-\exp(-\lambda_nt+b_n\tau+\frac{1}{2}\epsilon^2\tau)\}\\
&=&\sum
a_n^2\exp\{-\lambda_n(t+\tau)\}[\exp\{\epsilon^2(t\wedge\tau)\}-1]
\end{eqnarray*}
and
\begin{eqnarray*}
& & Corr(u(x,t), u(x,\tau))  \nonumber \\
&=&\frac{Cov(u(x,t),u(x,\tau))}{\sqrt{Var(u(x,t))} \;
\sqrt{Var(u(x,\tau))}}
\nonumber \\
&=& \frac{\sum
a_n^2\exp\{-\lambda_n(t+\tau)\}[\exp\{\epsilon^2(t\wedge\tau)\}-1]}{\sqrt{\sum
a_n^2 \exp(-2\lambda_n t ) [\exp(\epsilon^2 t)-1]} \; \sqrt{\sum
a_n^2 \exp(-2\lambda_n\tau )[\exp(\epsilon^2\tau)-1]}}.
\end{eqnarray*}



\section{Lyapunov Exponents}
\label{liapunov}

Lyapunov exponents are tools for quantifying growth or decay of
linear systems (e.g., PDEs or SPDEs). The following discussions
are from   \cite{CL, Kwiec}.

\subsection{A deterministic PDE system}

Let us first look at the following deterministic PDE:
\begin{equation}\label{eq:pde}
    \frac{\partial u}{\partial t} =  u_{xx} + \alpha u,
\end{equation}
\begin{equation}\label{eq:pdeIC}
    u(0,x) = f(x),
\end{equation}
\begin{equation}\label{eq:pdeBC}
    u(t,x) = 0,  \;  x \in \partial D
\end{equation}
where $ D = \{x: 0 \leq x \leq 1\} $ and the function $ f \in
L^{2}(0,1)$. An orthonormal basis for $L^{2}(0,1)$ is $ \{
e_{n}(x) \}, n=0,1,2,\ldots, \p_{xx} e_{j} = - \lambda_{j}e_{j},
$. Note that $0 \leq \lambda_{j} \uparrow \infty. $

We then can write:
\begin{equation}\label{eq:expf}
   f = \sum_{j=0}^{\infty} f_{j} e_{j}, \; \text{where} \;  f_{j} = \langle f,
   e_{j} \rangle.
\end{equation}
By using the method of eigenfunction expansion, it is known that
the unique solution to the problem is given below:
\begin{equation}\label{eq:solu}
u(t,x) = \sum_{j=0}^{\infty} \exp(t(-\lambda_{j} +
\alpha))f_{j}e_{j}(x), \;  t \geq 0.
\end{equation}

\begin{theorem}\label{thm1}
Let us fix an initial condition $f, f \neq 0$. Let $j_{0}$ be the
smallest integer $j \geq 0$ in the expansion \eqref{eq:expf} of
$f$ such that $f_{j_{0}} \neq 0$. Then the Lyapunov exponent of
the system $ \eqref{eq:pde} - \eqref{eq:pdeBC} $ exists as a limit
and is given by
\begin{equation}
   \lambda^{u}(f) = -\lambda_{j_{0}} + \alpha.
\end{equation}
\end{theorem}

\begin{proof}
For a class of initial conditions $f$ we calculate the Lyapunov
exponents, which are defined as

\begin{equation}
   \lambda^{u}(f) =  \limsup_{t\to\infty} \frac{1}{t} \log \|u(t)\|_{L^{2}}.
\end{equation}

By applying  \eqref{eq:solu}, we obtain the Lyapunov exponents
regarding to PDE system  $ \eqref{eq:pde}  -  \eqref{eq:pdeBC} $,
\begin{equation}
   \lambda^{u}(f) =  \limsup_{t\to\infty} \frac{1}{t} \log
   \left\|\sum_{j=0}^{\infty} \exp(t(-\lambda_{j} + \alpha))f_{j}e_{j}(x)\right\|.
\end{equation}
On the one hand,
\begin{eqnarray}
    \frac{1}{t} \log \left\| \sum_{j=0}^{\infty} \exp(t(-\lambda_{j} + \alpha))
    f_{j}e_{j}(x)\right\| & \leq &  \frac{1}{t} \log \left(\sum_{j=j_{0}}^{\infty} |
    \exp(t(-\lambda_{j_{0}} + \alpha))f_{j}|^{2}\right)^{1/2} \nonumber \\
    & = & -\lambda_{j_{0}} + \alpha + \frac{1}{t} \log \|f\|.
\end{eqnarray}
On the other hand,
\begin{eqnarray}
    \frac{1}{t} \log \left\| \sum_{j=0}^{\infty} \exp(t(-\lambda_{j} + \alpha))
    f_{j}e_{j}(x)\right\| & \geq &  \frac{1}{t} \log |
    \exp(t(-\lambda_{j-{0}} + \alpha))f_{j_{0}}| \nonumber \\
    & = &  -\lambda_{j_{0}} + \alpha + \frac{1}{t} \log |f_{j_{0}}|.
\end{eqnarray}
\end{proof}

\subsection{A SPDE system}

We now consider the following SPDE
\begin{equation}
    dv  =  ( v_{xx} + \beta v )dt + \gamma v \; dw_{t}, \label{eq:spde}
\end{equation}
\begin{equation}
    v(0,x,\omega) = f(x),   x \in D \label{eq:spdeIC}
\end{equation}
\begin{equation}
    v(t,x,\omega) = 0, \;   x \in \partial D, \label{eq:spdeBC}
\end{equation}
where $w_t$ is a scalar Brownian motion.  The conditions
\eqref{eq:spdeIC} and \eqref{eq:spdeBC} hold for a. a. $\omega \in
\Omega$.

We seek the solution  with expansion with respect to the basis
$\{e_{j}\}$ (see the last subsection)
\begin{equation}\label{eq:expsol}
   v(t,x) = \sum_{j=0}^{\infty} y_{j}(t) e_{j}(x),
\end{equation}
where $y_{j}(t)$, for $ j=0,1,2,\ldots $ satisfy the following
stochastic ordinary differential equations:
\begin{eqnarray}
    dy_{j}(t) & = & ( - \lambda_{j} + \beta )y_{j}(t)dt + \gamma y_{j}(t)
    dw_{t}, \label{eq:nspde}\\
    y_{j}(0) & = & f_{j}.  \nonumber
\end{eqnarray}
So
\begin{equation}
    y_{j}(t) = \exp(\gamma w_{t} ) \exp \left(\left( - \lambda_{j} +
    \beta - \frac{1}{2}\gamma^{2} \right)t \right) f_{j}. \label{eq:coef}
\end{equation}
Thus from  \eqref{eq:expsol}, we obtain,
\begin{equation}
    v(t,x) =  \sum_{j=0}^{\infty} \exp(\gamma w_{t} ) \exp \left(\left( - \lambda_{j} +
    \beta - \frac{1}{2}\gamma^{2} \right)t \right) f_{j}e_{j}. \label{eq:sol1}
\end{equation}
Observe that
\begin{equation}
    v(t,x) = \exp(\gamma w_{t} ) \exp \left(\left( (\beta
    - \alpha) \frac{1}{2}\gamma^{2} \right)t \right)u(t,x), \label{eq:so2}
\end{equation}
where $u(t,x)$ is the solution to the above deterministic PDE
\eqref{eq:pde} - \eqref{eq:pdeBC}.

By  \eqref{eq:so2}, we can calculate the Lyapunov exponent of the
stochastic system $ \eqref{eq:spde} -  \eqref{eq:spdeBC} $ as a
function of the Lyapunov exponent of the deterministic system $
\eqref{eq:pde}  -  \eqref{eq:pdeBC} $ as follows:
\begin{eqnarray}
    \lambda^{v}(f) & = & \limsup_{t\to\infty} \frac{1}{t} \log
    \|v(t)\| \nonumber\\
    & = & \limsup_{t\to\infty} \frac{1}{t} \log \left \|
    \exp(\gamma w_{t} ) \exp \left(\left((\beta - \alpha)-
    \frac{1}{2}\gamma^{2} \right)t \right)u(t) \right\| \nonumber \\
    & = & \lambda^{u}(f) + (\beta - \alpha)-
    \frac{1}{2}\gamma^{2}, \;  a.s.
\end{eqnarray}
  by the strong law of large number.

Let us state the   result in the following  theorem.
\begin{theorem}\label{thm2}
Let $f \neq 0$. Then the Lyapunov exponent of the   SPDE   $
\eqref{eq:spde}  -  \eqref{eq:spdeBC} $ almost surely exists as a
limit, is non-random and is given in the following formula:
\begin{equation}
    \lambda^{v}(f) = \lambda^{u}(f) + (\beta - \alpha)-
    \frac{1}{2}\gamma^{2}, \; a.s.
\end{equation}
\end{theorem}

\begin{remark}
Let us consider a special case when $ \alpha = \beta $. Then by
the above theorem, for a fixed initial condition $f$, the Lyapunov
exponent of the stochastic system $ \eqref{eq:spde} -
 \eqref{eq:spdeBC} $ is
\begin{equation}
    \lambda^{v}(f) = \lambda^{u}(f)- \frac{1}{2}\gamma^{2},
\end{equation}
which obviously is smaller than the Lyapunov exponent of the
 corresponding deterministic system $ \eqref{eq:pde} -
 \eqref{eq:pdeBC} $. The result implies that this stochastically perturbed system
is more stable than the original deterministic system.
\end{remark}

\section{Impact of Uncertainty }
\label{impact}

In this section, we first recall some inequalities for estimating
solutions of SPDEs, and then we estimate the impact of noises on
solutions of the nonlinear Burgers equation.

\subsection{Differential and integral inequalities}

\textbf{Gronwall inequality}: Differential form \cite{Tem97}

 Assuming that $y(t)\geq 0$, $g(t)$ and $h(t)$
are integrable, if $\frac{dy}{dt}\leq g(t) y +h(t)$ for $t \geq
t_0$, then
\begin{equation*}
y(t)\leq y(t_0)e^{\int_{t_0}^t g(\tau)d\tau} +\int_{t_0}^t h(s)
[e^{ \int_s^t g(\tau)d\tau}] ds, \;\; t \geq t_0.
\end{equation*}

In particular, if $\frac{dy}{dt}\leq g y +h$ for $t \geq t_0$ with
$g, h$ being constants and $t_0=0$, we have
\begin{equation*}
y(t)\leq y( 0)e^{g t} - \frac{h}{g} (1- e^{ gt}), \;\; t \geq 0.
\end{equation*}
Note that when constant $g<0$, then $\lim_{t\to \infty} y(t)=-
\frac{h}{g}$.

\medskip

\textbf{Gronwall inequality}: Integral form \cite{C/L, GH}

If $u(t), v(t)$ and $c(t)$ are all non-negative, $c(t)$ is
differentiable,  and $v(t)  \leq c(t) + \int_0^t u(s) v(s)ds$ for
$t \geq t_0$, then
\begin{equation*}
v(t)\leq v(t_0)e^{\int_{t_0}^t u(\tau)d\tau } +\int_{t_0}^t c'(s)[
e^{ \int_s^t u(\tau)d\tau}] ds, \;\; t \geq t_0.
\end{equation*}

In particular, assuming that $y(t)\geq 0$ and is continuous and $
y(t) \leq C + K \int_0^t y(s)ds$, with $C, K$ being positive
constants, for $t>0$. Then
\begin{equation*}
y(t)\leq  C e^{Kt}, \;\; t \geq 0.
\end{equation*}

\subsection{Sobolev inequalities}

We first introduce some common Sobolev spaces.
 For $k=1, 2, \cdots$, we define
 $$H^k (0, l) := \{f: \; f, f', \cdots, f^{(k)} \in L^2(0,l)  \}$$
Each of these is a Hilbert spaces with the scalar product
$$
\langle u, v\rangle_k = \int_0^l [uv+u'v'+\cdots +u^{(k)}v^{(k)}]
dx,
$$
and the norm
$$
\|u\|_k = \sqrt{\langle u, u\rangle_k} = \sqrt{\int_0^l [u^2+
(u')^2+\cdots + (u^{(k)})^2] dx}.
$$

For $k=1, 2, \cdots$ and $p \geq 1$, we further define another
class of  Sobolev spaces
$$
 W^{k, p}(D) =\{u: \; u, Du, \cdots, D^{\a}u \in L^p(D), |\a| \leq k
 \},
$$
with norm
$$
\|u\|_{k,p} =(\|u\|^p_{L^p} + \|u'\|^p_{L^p}+\cdots +
\|u^{(k)}\|^p_{L^p})^{\frac1{p}}.
$$
Moreover, $H_0^k(0, l)$ denotes the closure of $C_c^{\infty}(0,
l)$ in $H^k(0,l)$ (i.e., under the norm $\|\cdot\|_k$). It is a
sub-Hilbert space in $H^k(0,l)$. Similarly, $W_0^{k, p}(0, l)$
denotes the closure of $C_c^{\infty}(0, l)$ in $W^{k, p}(0,l)$
(i.e., under the norm $\|\cdot\|_{k,p}$). It is a sub-Hilbert
space in $W^{k, p}(0,l)$.

Standard abbreviations $L^2$ $=$ $L^2(D)$, $H^k_0$ $=$ $H^k_0(D)$,
$k = 1, 2,\ldots$,  are used for the common Sobolev spaces in
fluid mechanics, with $<\cdot, \cdot>$ and $\| \cdot \|$ denoting
the  usual (spatial) scalar product and norm, respectively, in
$L^2(D)$:
$$
<f, g>:=\int_D fg dxdy, \;\;\;\;\;\;  \|f\|:=\sqrt{<f, f>}
=\sqrt{\int_D f(x,y)dxdy}.
$$

\textbf{Cauchy-Schwarz inequality: }

In the   space $L^2(D)$ of square-integrable   functions defined
on a domain $D \subset \R^n$:
\begin{equation*}
 |\int_D f(x) g(x) dx|  \leq \sqrt{\int_D f^2(x)dx }
\; \sqrt{\int_D g^2(x)dx}.
\end{equation*}

\textbf{H\"older inequality}:

In the   space $L^r(D)$ of
     functions defined on a domain $D \subset \R^n$:
\begin{equation*}
 |\int_D f(x) g(x) dx|  \leq (\int_D |f(x)|^p dx)^{\frac1{p}}
\; (\int_D |g(x)|^q dx)^{\frac1{q}}.
\end{equation*}

\textbf{Minkowski inequality: }

In the   space $L^p(D)$ of
     functions defined on a domain $D \subset \R^n$:
\begin{equation*}
 (\int_D |f(x) \pm g(x)|^p dx)^{\frac1{p}}   \leq (\int_D |f(x)|^p dx)^{\frac1{p}}
+ (\int_D |g(x)|^p dx)^{\frac1{p}}.
\end{equation*}

\textbf{Poincar\'e inequality} \cite{Tem97}:

For $g \in H^1_0(D)$,
\begin{eqnarray*}
\|g\|^2 = \int_D g^2(x,y) \,dxdy \leq  \frac{|D|}{\pi} \int_D
|\nabla g|^2 \,dxdy = \frac{|D|}{\pi} \|\nabla g\|^2 ,
\end{eqnarray*}
where $|D|$ is the Lebesgue measure of the domain $D$.

 For $u\in W^{1, p}_0(D)$, $1 \leq p < \infty$ and
$D\subset \R^n$ a bounded domain
\begin{eqnarray*}
\|u\|_p    \leq  C \| \n u\|_p,
\end{eqnarray*}
where $C$ is a positive constant depending only on the domain $D$.

 Let $u\in W^{1, p}(D)$, $1 \leq p < \infty$ and
$D\subset \R^n$ a bounded convex domain. Let $S \subset D$ be a
measurable subset, and define the spatial average of $u$ over $S$
by $ u_S = \frac1{|S|} \int_D u dx$ (here $|S|$ is the volume or
Lebesgue measure of $S$). Then
\begin{eqnarray*}
\|u -u_S\|_p    \leq  C \| \n u\|_p,
\end{eqnarray*}
where $C$ is a positive constant depending only on the domain $D$
and $S$.

\textbf{Agmon inequality} \cite{Tem97}:

 Let $D \subset \R^n$. There exists a constant $C$
depending only on domain $D$ such that
\begin{eqnarray*}
\|u\|_{L^{\infty}(D)} & \leq &  C \|
u\|^{\frac12}_{H^{\frac{n-1}2}(D)}
 \;\| u\|^{\frac12}_{H^{\frac{n+1}2}(D)},\;  \mbox{for  n  odd}, \\
\|u\|_{L^{\infty}(D)} & \leq &
\|u\|^{\frac12}_{H^{\frac{n-2}2}(D)} \;
\|u\|^{\frac12}_{H^{\frac{n+2}2}(D)},\;  \mbox{for  n  even}.
\end{eqnarray*}

In particular, for $n=1$ and $u \in H^1 (0, l)$,
\begin{eqnarray*}
\|u\|_{L^{\infty}(0, l)} \leq C \| u\|^{\frac12}_{L^2 (0, l)}
 \;\| u\|^{\frac12}_{H^1 (0, l)}.
\end{eqnarray*}

Moreover, for $n=1$ and $u \in H^1_0 (0, l)$,
\begin{eqnarray*}
\|u\|_{L^{\infty}(0, l)} \leq C \| u\|^{\frac12}_{L^2 (0, l)}
 \;\| u_x\|^{\frac12}_{L^2 (0, l)}.
\end{eqnarray*}

\subsection{Stochastic Burgers equation }

 We now consider the Burgers equation with additive   noise
 forcing as in \cite{BlomkerDuan}:
\begin{equation}
  \partial_t u +u\cdot \partial_x u =\nu \partial_x^2 u + \sigma\dot{W}_t
\end{equation}
\begin{equation}
    u(\cdot,0) = 0,
  \qquad   u(\cdot,l) =0, \qquad u(x,0)=u_0(x),
\end{equation}
where $W_t$ is a Brownian motion, with covariance $Q$, taking
values in the Hilbert space $L^2(0, l)$ with the usual scalar
product $\langle \cdot, \cdot \rangle$. We assume that the trace
$Tr (Q)$ is finite. So $\dot{W}_t$ is noise colored in space but
white in time.

Taking $F(u)=\frac12 \int_0^l u^2 dx =\frac12 \langle u, u
\rangle$ and applying the Ito's formula, we obtain
\begin{eqnarray}
   \frac12 d\|u\|^2 = \langle u, \sigma dW_t \rangle
   + [\langle u, \nu u_{xx}-uu_x \rangle + \frac12 \sigma^2\; l \;Tr(Q)] dt.
\end{eqnarray}
Thus
\begin{eqnarray}
    \ddt \EX  \|u\|^2
&=&  2 \EX \langle u, \nu u_{xx}-uu_x \rangle +  \sigma^2 \; l\;  Tr(Q) \nonumber  \\
&=& -2 \nu  \EX \|u_x\|^2  +  \sigma^2\; l\; Tr(Q).
\end{eqnarray}
By the Poincare inequality $\|u\|^2 \leq  c\|u_x\|^2$ for some
positive constant depending only on the interval $(0, l)$, we have
\begin{eqnarray}
    \ddt \EX \|u\|^2  \leq  -\frac{2\nu}{c} \EX \|u\|^2   +  \sigma^2\; l\;  Tr(Q).
\end{eqnarray}
Then using the Gronwall inequality, we finally get
\begin{eqnarray} \label{square-estimate}
      \EX \|u\|^2  \leq \EX \|u_0\|^2 e^{-\frac{2\nu}{c}t}
        +  \frac12 c \; \sigma^2 l \; Tr(Q) [1 - e^{-\frac{2\nu}{c}t}].
\end{eqnarray}
Note that the first term in this estimate involves the initial
data, and the second term involves   the noise intensity $\sigma$
as well as the trace of the noise covariance.

\bigskip

We finally consider the Burgers equation with multiplicative noise
forcing.

\begin{equation}
  \partial_t u +u\cdot \partial_x u =\nu \partial_x^2 u + \sigma u
  \dot{w}_t,
\end{equation}
with the same boundary condition and initial condition as above,
where $w_t$ is a scalar Brownian motion (e.g.,  with covariance
$Q=1$ and the trace $Tr (Q)=1$). So $\dot{W}_t$ is noise
homogeneous in space but white in time.

By the Ito's formula, we obtain
\begin{eqnarray}
   \frac12 d\|u\|^2 = \langle u, \sigma u dw_t \rangle
   + [\langle u, \nu u_{xx}-uu_x \rangle + \frac12 \sigma^2 \|u\|^2] dt.
\end{eqnarray}
Thus
\begin{eqnarray}
    \ddt  \EX \|u\|^2
&=&  2 \EX \langle u, \nu u_{xx}-uu_x \rangle +  \sigma^2 \EX \|u\|^2  \nonumber  \\
&=& -2 \nu \EX \|u_x\|^2  +  \sigma^2 \EX \|u\|^2  \nonumber \\
&\leq &  (\sigma^2-\frac{2\nu}{c}) \EX \|u\|^2.
\end{eqnarray}
Therefore,
\begin{eqnarray}
    \EX \|u\|^2   \leq  \EX \|u_0\|^2 e^{(\sigma^2-\frac{2\nu}{c}) t} .
\end{eqnarray}

Note here that the multiplicative noise affects the mean energy
growth or decay rate, while the additive noise affects the mean
energy upper bound.

\subsection{Likelihood for staying bounded}

By the Chebyshev inequality, we can estimate the likelihood of
solution orbits staying inside or outside a bounded domain in
Hilbert space $H=L^2(0, l)$. Taking the bounded domain as a ball
centered at the origin with radius   $\delta
>0$. For example, for the above   Burgers equation with multiplicative noise,
we have
\begin{eqnarray}
  \PX (\om:  \|u\| \geq \delta) &\leq & \frac1{\delta^2}  \EX \|u\|^2
  \nonumber  \\
  & \leq &  \frac{\EX \|u_0\|^2}{\delta^2} \; e^{(\sigma^2-\frac{2\nu}{c}) t} .
\end{eqnarray}
and
\begin{eqnarray}
  \PX (\om:  \|u\| < \delta) &=& 1-\PX (\om:  \|u\| \geq \delta)
  \nonumber \\
   &\geq &  1-\frac{\EX \|u_0\|^2}{\delta^2} \; e^{(\sigma^2-\frac{2\nu}{c}) t} .
\end{eqnarray}

{\bf Acknowledgements.} I would like to thank Hongbo Fu and Jiarui
Yang for helpful comments. This work was partly supported by the
NSF Grants 0542450 and 0620539.


\end{document}